\documentclass[letterpaper, reqno,11pt]{article}
\usepackage[margin=1.0in]{geometry}
\usepackage{color,latexsym,amsmath,amssymb, amsthm, graphicx,subfigure,revsymb, enumerate}
\usepackage[percent]{overpic}

\numberwithin{equation}{section}

\newcommand{\RR}{\mathbb{R}}

\newcommand{\eps}{\varepsilon}
\newcommand{\tubes}{\mathbb{T}}

\newtheorem{thm}{Theorem}[section]
\newtheorem{lem}[thm]{Lemma}

\newtheorem{cor}[thm]{Corollary}
\newtheorem{conj}[thm]{Conjecture}

\theoremstyle{remark}

\makeatletter
\newcommand{\myitem}[1]{%
\item[#1.]\protected@edef\@currentlabel{#1}%
}
\makeatother

\begin{document}
\pagenumbering{arabic}
\title{Unions of lines in $\mathbb{R}^n$}

\author{Joshua Zahl\thanks{University of British Columbia, Vancouver BC, supported by an NSERC Discovery grant, jzahl@math.ubc.ca.
}}

\maketitle

\begin{abstract}
\noindent We prove a conjecture of D.~Oberlin on the dimension of unions of lines in $\RR^n$. If $d\geq 1$ is an integer, $0\leq\beta\leq 1$, and $L$ is a set of lines in $\RR^n$ with Hausdorff dimension at least $2(d-1)+\beta$, then the union of the lines in $L$ has Hausdorff dimension at least $d + \beta$. Our proof combines a refined version of the multilinear Kakeya theorem by Carbery and Valdimarsson with the multilinear $\to$ linear argument of Bourgain and Guth.
\end{abstract}

\section{Introduction}\label{introSection}
This paper concerns the following conjecture of D.~Oberlin. In what follows, $\dim$ will always refer to Hausdorff dimension.
\begin{conj}\label{mainConj}
Let $1\leq d\leq n$ be integers and let $\beta\in [0,1]$. Let $L$ be a set of lines in $\RR^n$, with $\dim L \geq 2(d-1)+\beta$. Then
\begin{equation}\label{dimensionBound}
\dim \bigcup_{\ell \in L} \ell \geq d+\beta.
\end{equation}
\end{conj}
The bound \eqref{dimensionBound} is best-possible, since $L$ could be the set of lines contained in a $\beta$-dimensional union of (parallel) $d$-dimensional affine subspaces of $\RR^n$. In \cite{ROb}, R.~Oberlin proved the finite field analogue of Conjecture \ref{mainConj}, and this is also the first place where the conjecture appeared in print. In \cite{HKM}, H\'era, Keleti, and M\'ath\'e considered a variant of Conjecture \ref{mainConj} where $d=0$ and $L$ is replaced by a set of $k$-dimensional affine subspaces of $\RR^n$. Finally in \cite{He}, H\'era posed a generalization of Conjecture \ref{mainConj}, where lines are replaced by $k$-dimensional affine subspaces of $\RR^n$. 

When $d=1$, Conjecture \ref{mainConj} was proved by H\'era, Keleti, and M\'ath\'e \cite{HKM} and when $d=n-1,$ Conjecture \ref{mainConj} was proved by D.~Oberlin \cite{DOb}. In this paper, we will prove Conjecture \ref{mainConj} for the remaining values of $d$. As is typical for such problems, the conjecture will follow from the corresponding maximal function estimate, which we will state below. In what follows, $\mathcal{L}_n$ is the set of lines in $\RR^n$, with the following metric: if $\ell,\ell'\in\mathcal{L}_n$, then $d(\ell,\ell') = |x-x'| + |u\wedge u'|$. Here $u$ (resp $u'$) is a unit vector parallel to $\ell$; $x$ (resp. $x'$) is the unique point on $\ell$ with $x\perp u$; and $|u\wedge u|$ is the (unsigned) area of the parallelogram spanned by $u$ and $u'$. A $\delta$-tube is the $\delta$-neighborhood of a unit line segment in $\RR^n$, and if $T$ is a $\delta$-tube, then $\ell(T)$ is the line coaxial with $T$. Finally, if $p\in [1,\infty]$, then $p'$ denotes the conjugate Lebesgue exponent.
\begin{thm}\label{maximalFnEstimateTubes}
Let $n\geq 2$ and $\eps>0$. Then there exists $C=C(n,\eps)$ so that the following holds. Let $1\leq d<n$ be an integer, and let $0\leq\beta\leq 1$. Let $\tubes$ be a set of $\delta$-tubes in $B(0,1)\subset \RR^n$. Suppose that for all balls $B\subset\mathcal{L}_n$ of radius $r$, we have
\begin{equation}\label{nonConcentration}
\#\{T\in\tubes \colon \ell(T)\subset B\}\leq (r/\delta)^{2(d-1)+\beta}.
\end{equation}
Then
\begin{equation}\label{LPBound}
\Big\Vert \sum_{T\in\tubes}\chi_T\Big\Vert_p \leq C\delta^{\frac{1-d}{p'}-\eps}\Big(\sum_{T\in\tubes}|T|\Big)^{1/p},\quad p' = d+\beta.
\end{equation} 
\end{thm}
We prove Theorem \ref{maximalFnEstimateTubes} by combining the multilinear Kakeya theorem \cite{BCT, G} with the multilinear to  linear arguments developed by Bourgain and Guth in \cite{BG}; this is done in Section \ref{BGSection}. A secondary goal of this paper is provide a short and self-contained introduction to this type of argument. We prove Conjecture \ref{mainConj} by first using a result from \cite{HKM} to reduce to the case where $L$ is analytic and hence supports a Frostman measure, and then using a standard discretization argument; this is done in Section \ref{mainConjProofSec}.

\subsection{Thanks}
The author would like to thank Korn\'elia H\'era for comments on an earlier version of this manuscript. 

\section{The Bourgain-Guth multilinear $\to$ linear argument}\label{BGSection}
In \cite{BG}, Bourgain and Guth combined the multilinear restriction theorem of Bennett, Carbery, and Tao \cite{BCT} (see also \cite{G}) with induction on scales to prove new Fourier restriction estimates. In this section, we revisit those arguments and obtain variants that are adapted to the proof of Theorem \ref{maximalFnEstimateTubes}. 

In the arguments that follow, we adopt the notation $A\lesssim B$ to mean that there is a constant $C$ (which may depend on the ambient dimension $n$, and possibly also the Lebesgue exponent $p$, if this is relevant) so that $A\leq CB$. 

We will use the following version of the multilinear Kakeya theorem, which is Theorem 1 from \cite{CV}. In what follows, $|u_1\wedge\ldots\wedge u_k|$ is the unsigned volume of the parallelepiped spanned by $u_1,\ldots,u_k$.
\begin{thm}\label{multilinearKakeyaThm}
Let $2\leq k\leq n$ and let $\tubes_1,\ldots,\tubes_k$ be sets of $\delta$-tubes in $\RR^n$. Then
\begin{equation}
\Big\Vert\Big(\sum_{T_1\in\tubes_1}\ldots\sum_{T_k\in\tubes_k} \chi_{T_1}\cdots\chi_{T_k}|u_1\wedge\ldots\wedge u_k|\Big)^{\frac{1}{k}}\Big\Vert_{\frac{k}{k-1}}
\lesssim \Big(\frac{1}{\delta}\Big)^{\frac{n}{k}-1}\prod_{i=1}^k\Big(\sum_{T_i\in\tubes_i}|T_i|\Big)^{\frac{1}{k}},
\end{equation}
where in the above expression, $u_i$ is the direction of the tube $T_i$. 
\end{thm}

When proving Theorem \ref{maximalFnEstimateTubes}, Theorem \ref{multilinearKakeyaThm} is helpful in the special case where, for a typical point $x$ where $\big(\sum\chi_T\big)^p$ is large, most $(d+1)$-tuples of tubes $T_1,\ldots,T_{d+1}$ that contain $x$ satisfy $|u_1\wedge\ldots\wedge u_{d+1}|\sim 1$. Indeed, it is straightforward to verify that our desired bound \eqref{LPBound} follows from Theorem \ref{multilinearKakeyaThm} in this case. The Bourgain-Guth multilinear $\to$ linear argument establishes a dichotomy between this situation and an opposite extreme, where most of the tubes containing a typical point $x$ are contained in the thin neighborhood of a $d$-dimensional affine subspace. 

The lemmas below will help make this dichotomy precise. In what follows, if $H$ is a $k$-dimensional vector subspace of $\RR^n$ and $u\in\RR^n$ is a unit vector, we write $|H\wedge u|$ to mean $|u_1\wedge \ldots\wedge u_k\wedge u|$, where $u_1,\ldots,u_k$ are orthonormal vectors that span $H$; the value of $|H\wedge u|$ is independent of the choice of $u_1,\ldots,u_k$. In particular, $|H\wedge u|=0$ if $u\in H$, and $|H\wedge u|=1$ if $u\perp H$.

\begin{lem}\label{kLinearVsPlainy}
Let $U\subset S^{n-1}$ be a multiset of unit vectors, let $1\leq k\leq n$, and let $0\leq \rho\leq 1$. Then at least one of the following must hold.

\medskip
\noindent {\bf A.}
\vspace*{-0.5cm}
\[
\#\big\{(u_1,\ldots,u_k)\in U^k\colon |u_1\wedge\ldots\wedge u_k|\geq\rho^{k-1}\big\} \geq \frac{1}{2}(\#U)^k.
\]
\noindent {\bf B.} There is a $(k-1)$-dimensional vector subspace $H\subset \RR^n$ so that
\[
\#\{u\in U\colon |H\wedge u|\leq \rho\} \geq 2^{-2k}(\#U).
\]

In particular, we always have
\begin{equation}\label{controlCardU}
(\#U)\lesssim \rho^{\frac{1-k}{k}}\Big(\sum_{u_1\in U}\ldots\sum_{u_k\in U}|u_1\wedge\ldots \wedge u_k|\Big)^{1/k}+\sup_{H\in \operatorname{Gr}(k-1,\RR^n)}\#\{u\in U\colon |H\wedge u|\leq\rho\}.
\end{equation}
\end{lem}
\begin{proof}
When $k=1$ there is nothing to prove, so we can assume $k\geq 2$. Suppose Option A does not hold. Let $N = \#U$ and let $2\leq j\leq k$ be the smallest index so that
\begin{equation}\label{jTuples}
\#\big\{(u_1,\ldots,u_j)\in U^j\colon |u_1\wedge\ldots\wedge u_j|<\rho^{j-1}\big\}\geq 2^{2j-2k-1}N^j.
\end{equation}
(Such an index must exist, since the failure of Option A implies that \eqref{jTuples} holds for $j=k$). For notational simplicity, we will suppose that $j>2$ (the same argument works when $j=2$, though some care must be taken with indices). There are at least $ 2^{2j-2k-2}N^{j-1}$ tuples $(u_1,\ldots,u_{j-1})$ for which
\begin{equation}\label{smallWedge}
\#\big\{u\in U\colon |u_1\wedge\ldots\wedge u_{j-1}\wedge u| < \rho^{j-1}\big\}\geq 2^{2j-2k-2}N.
\end{equation} 
But by the minimality of $j$, at most $2^{2j-2k-3}N^{j-1}$ of these tuples satisfy $|u_1\wedge\ldots\wedge u_{j-1}|<\rho^{j-2}$. We conclude that there exists a tuple $(u_1,\ldots u_{j-1})$ that satisfies \eqref{smallWedge}, with $|u_1\wedge\ldots\wedge u_{j-1}|\geq \rho^{j-2}$. Let $W=\operatorname{span}\{u_1,\ldots,u_{j-1}\}$. Then $|W\wedge u|\leq \rho$ for each $u$ in the set \eqref{smallWedge}, i.e.
\[
\#\{u\in U\colon |W\wedge u|\leq \rho\}\geq 2^{2j-2k-2}N \geq 2^{-2k}N.
\]
The result now follows by selecting a $(k-1)$-dimensional vector subspace $H$ containing $W$. 
\end{proof}

\begin{cor}\label{LpConsequence}
Let $U\subset S^{n-1}$ be a multiset of unit vectors, let $2\leq k\leq n$, and let $0<\rho\leq 1$. Let $\{\tau\}$ be a covering of $S^{n-1}$ by spherical caps of diameter $\rho$, and suppose each $u\in S^{n-1}$ is contained in at most $10^{n}$ caps. Let $1\leq p < \infty$. Then
\begin{equation}\label{controlUp}
(\#U)^p \lesssim \rho^{\frac{(1-k)p}{k}}\Big(\sum_{u_1\in U}\ldots\sum_{u_k\in U}|u_1\wedge\ldots \wedge u_k|\Big)^{p/k}+\rho^{(2-k)(p-1)}\sum_\tau\big(\# (U\cap\tau)\big)^p.
\end{equation}
\end{cor}
\begin{proof}
Apply Lemma \ref{kLinearVsPlainy}. The first term in \eqref{controlCardU} corresponds to the first term in \eqref{controlUp}. For the second term, note that if $H$ is a $(k-1)$-dimensional subspace of $\RR^n$, then $\lesssim \rho^{2-k}$ of the caps from $\{\tau\}$ intersect the set $\{u\in S^{n-1}\colon |H\wedge u|\leq \rho\}$ (it is not important whether the supremum in \eqref{controlCardU} is realized). Next we use H\"older's inequality, and finally we add the remaining caps to the RHS of \eqref{controlUp}. 
\end{proof}

We record the following consequence of Corollary \ref{LpConsequence}.
\begin{lem}
Let $0<\delta\leq\rho\leq 1$, let $2\leq k\leq n$, and let $1\leq p<\infty$. 
 Let $\{\tau\}$ be a covering of $S^{n-1}$ by spherical caps of diameter $\rho$ that is at most $10^{n}$-overlapping. Let $\tubes$ be a set of $\delta$-tubes in $B(0,1)\subset \RR^n$. Then
\begin{equation}\label{decomposeLpNorm}
\Big\Vert \sum_{T\in\tubes}\chi_T\Big\Vert_p 
\lesssim \rho^{\frac{1-k}{k}}\Big\Vert \Big(\sum_{T_1,\ldots,T_k \in\tubes} \chi_{T_1}\cdots\chi_{T_k} |u_1\wedge\ldots \wedge u_k|\Big)^{\frac{1}{k}}\Big\Vert_p
+\rho^{\frac{2-k}{p'}}\Big(\sum_\tau\Big\Vert \sum_{\substack{T\in\tubes\\ \operatorname{dir}(T)\in\tau}}\chi_T  \Big\Vert_p^p\Big)^{\frac{1}{p}}.
\end{equation}
\end{lem}
\begin{proof}
For each $x\in\RR^n$, we apply Corollary \ref{LpConsequence} to the multiset of unit vectors $\{\operatorname{dir}(T)\colon x\in T\}$ and integrate. 
\end{proof}
Observe that if $\delta<\rho/2$, then the final term in \eqref{decomposeLpNorm} can be further decomposed. For each cap $\tau$ with center $u\in S^{n-1}$, we can cover $B(0,1)\subset\RR^n$ by boundedly overlapping $\rho$-tubes that point in direction $u$. If $T\in\tubes$ and $\operatorname{dir}(T)\in \tau$, then $T$ will be contained in at least 1, and at most $\lesssim 1$ tubes $T_\rho$ of this type. Let $\tubes_\rho$ denote the set of all such tubes, for all caps $\tau$. Then each $T\in\tubes$ is contained in at least 1, and at most $\lesssim 1$ tubes from $\tubes_\rho$. Thus we have
\begin{equation}\label{furtherDecomposition}
\sum_\tau\Big\Vert \sum_{\substack{T\in\tubes\\ \operatorname{dir}(T)\in\tau}}\chi_T  \Big\Vert_p^p
\lesssim \sum_{T_\rho\in\tubes_\rho}\Big\Vert \sum_{\substack{T\in\tubes\\ T\subset T_\rho }}\chi_T  \Big\Vert_p^p,\quad 1\leq p<\infty.
\end{equation}
The advantage is that each term on the RHS of \eqref{furtherDecomposition} has been localized to a $\rho$-tube.

We will use Theorem \ref{multilinearKakeyaThm} to control the first term on the RHS of \eqref{decomposeLpNorm}. We record the following computation.
\begin{lem}\label{calculationLem}
Let $1\leq d<n$ be integers, let $\beta\in[0,1]$, and let $\tubes$ be a set of $\delta$-tubes in $B(0,1)\subset \RR^n$, with
$\#\tubes\leq \delta^{2(1-d)-\beta}$. Then
\begin{equation}
\Big\Vert \Big(\sum_{T_1,\ldots,T_{d+1} \in\tubes} \chi_{T_1}\cdots\chi_{T_{d+1}} |u_1\wedge\ldots \wedge u_{d+1}|\Big)^{\frac{1}{d+1}}\Big\Vert_p\lesssim \delta^{\frac{1-d}{p'}}\Big(\sum_{T\in\tubes}|T|\Big)^{1/p},\quad p' = d+\beta.
\end{equation}
\end{lem}
\begin{proof}
Note that $p - \frac{d+1}{d}\geq 0$, with equality if $\beta = 1$. We have
\begin{equation}
\begin{split}
\int \Big(&\sum_{T_1,\ldots,T_{d+1} \in\tubes} \chi_{T_1}\cdots\chi_{T_{d+1}} |u_1\wedge\ldots \wedge u_{d+1}|\Big)^{\frac{p}{d+1}}\\
&\leq (\#\tubes)^{p - \frac{d+1}{d}} \int \Big(\sum_{T_1,\ldots,T_{d+1} \in\tubes} \chi_{T_1}\cdots\chi_{T_{d+1}} |u_1\wedge\ldots \wedge u_{d+1}|\Big)^{\frac{1}{d}}\\
&\lesssim \Big(\delta^{1-n}\sum_{T\in\tubes}|T|\Big)^{p - \frac{d+1}{d}}\Big[ \delta^{\frac{d+1-n}{d}} \Big(\sum_{T\in\tubes}|T|\Big)^{\frac{d+1}{d}}\Big]\\
& = \delta^{1+(1-n)(p-1)}\Big(\sum_{T\in\tubes}|T|\Big)^{p-1}\Big(\sum_{T\in\tubes}|T|\Big)\\
& \leq \delta^{1+(1-n)(p-1)}\Big(\delta^{n-1}\delta^{2(1-d)-\beta}\Big)^{p-1}\Big(\sum_{T\in\tubes}|T|\Big)\\
& = \delta^{\frac{p(1-d)}{p'}}\Big(\sum_{T\in\tubes}|T|\Big).
\end{split}
\end{equation}
For the first inequality, we used the observation that the integrand is bounded by $(\#\tubes)^{d+1}$. The second inequality used Theorem \ref{multilinearKakeyaThm}. The next equality re-groups terms. The fifth line used the hypothesis on the cardinality of $\tubes$, while the final line simplifies the expression. 
\end{proof}

Combining \eqref{furtherDecomposition} and Lemma \ref{calculationLem}, we obtain the following; recall that $\tubes_\rho$ is a maximal collection of boundedly overlapping $\rho$-tubes. 
\begin{cor}\label{towardsInductionOnScalesCor}
Let $1\leq d<n$ be integers and let $\beta\in[0,1]$. Let $0<\delta\leq \rho/2\leq 1$, and let $\tubes$ be a set of $\delta$-tubes in $B(0,1)\subset \RR^n$, with $\#\tubes\leq \delta^{2(1-d)-\beta}$. Then
\begin{equation}\label{towardsInductionOnScalesEqn}
\Big\Vert \sum_{T\in\tubes}\chi_T\Big\Vert_p 
\lesssim \rho^{\frac{-d}{d+1}} \delta^{\frac{1-d}{p'}}\Big(\sum_{T\in\tubes}|T|\Big)^{1/p} + \rho^{\frac{1-d}{p'}}\bigg(\sum_{T_\rho\in\tubes_\rho}\Big\Vert \sum_{\substack{T\in\tubes\\ T\subset T_\rho }}\chi_T  \Big\Vert_p^p\bigg)^{1/p},\quad p' = d+\beta.
\end{equation}

\end{cor}
Using Inequality \ref{towardsInductionOnScalesEqn}, we can prove Theorem \ref{maximalFnEstimateTubes} using induction on scales. 

\begin{proof}[Proof of Theorem \ref{maximalFnEstimateTubes}]
In the arguments that follow, we will choose the (small, positive) quantities $\rho$ and $\delta_0,$ and then the (large) quantity $C$, in that order; $C$ is the quantity appearing in \eqref{LPBound}.

We will prove the theorem by induction on $\delta$. By choosing $C$ appropriately, we can suppose that \eqref{LPBound} holds whenever $\delta\geq\delta_0$. Suppose now that $\delta\in (0,\delta_0)$, and suppose furthermore that for all $\tilde\delta\in (\delta,1]$, we have that \eqref{LPBound} has been established for all collections $\tilde\tubes$ of $\tilde\delta$-tubes in $B(0,1)\subset\RR^n$ that satisfy the analogue of \eqref{nonConcentration} with $\tilde\delta$ in place of $\delta$. Let $\tubes$ be a set of $\delta$-tubes in $B(0,1)$ that satisfy \eqref{nonConcentration}. Our goal is to establish \eqref{LPBound} for this collection of tubes.

By Corollary \ref{towardsInductionOnScalesCor}, we have
\begin{equation}\label{boundingLpNormP}
\Big\Vert \sum_{T\in\tubes}\chi_T\Big\Vert_p^p  \lesssim  \rho^{\frac{-dp}{d+1}} \delta^{\frac{(1-d)p}{p'}}\Big(\sum_{T\in\tubes}|T|\Big) + \rho^{\frac{(1-d)p}{p'}}\sum_{T_\rho\in\tubes_\rho}\Big\Vert \sum_{\substack{T\in\tubes\\ T\subset T_\rho }}\chi_T  \Big\Vert_p^p.
\end{equation}
If $C=C(n,\eps)$ is chosen sufficiently large depending on $\rho$ (which in turn, will ultimately only depend on $n$ and $\eps$) and the implicit constant in \eqref{boundingLpNormP}, then \eqref{boundingLpNormP} may be re-written as
\begin{equation}\label{boundingLpNormPRd2}
\Big\Vert \sum_{T\in\tubes}\chi_T\Big\Vert_p^p  - \frac{1}{2}C^p\delta^{\frac{(1-d)p}{p'}}\Big(\sum_{T\in\tubes}|T|\Big)\lesssim 
\rho^{\frac{(1-d)p}{p'}}\sum_{T_\rho\in\tubes_\rho}\Big\Vert \sum_{\substack{T\in\tubes\\ T\subset T_\rho }}\chi_T  \Big\Vert_p^p.
\end{equation}
We must now bound the RHS of \eqref{boundingLpNormPRd2}. Let $T_\rho\in\tubes_\rho$. Let $A$ be the affine map that sends the coaxial line of $T_\rho$ to the $e_1$-axis, and dilates $T_\rho$ by a factor of $1/\rho$ in each of the directions $e_2,\ldots,e_n$. We can verify that the image of each of the tubes $\{T\in\tubes\colon T\subset T_\rho\}$ under this transformation is comparable to a $\delta/\rho$-tube contained in $B(0,1)$, and that this set of $\delta/\rho$-tubes satisfies the analogue of \eqref{nonConcentration} with $\delta/\rho$ in place of $\delta$ (the inequality \eqref{nonConcentration} might be weakened by a constant factor, but this is harmless and can be remedied, for example, by partitioning the set of $\delta/\rho$-tubes into $\sim 1$ families, applying the arguments below to each family in turn, and using the triangle inequality). 

Since $\delta/\rho > \delta$, we can apply the induction hypothesis to conclude that (after undoing the re-scaling)
\[
\Big\Vert \sum_{\substack{T\in\tubes\\ T\subset T_\rho }}\chi_T  \Big\Vert_p^p \lesssim  C^p(\delta/\rho)^{\frac{(1-d)p}{p'}-\eps p}\Big(\sum_{\substack{T\in\tubes\\ T\subset T_\rho }}|T|\Big).
\]
Thus
\begin{equation}\label{boundSecondTerm}
\begin{split}
\rho^{\frac{(1-d)p}{p'}}\sum_{T_\rho\in\tubes_\rho}\Big\Vert \sum_{\substack{T\in\tubes\\ T\subset T_\rho }}\chi_T  \Big\Vert_p^p&\lesssim  C^p \rho^{\frac{(1-d)p}{p'}}(\delta/\rho)^{\frac{(1-d)p}{p'}-\eps p}\sum_{T_\rho\in\tubes_\rho}\Big(\sum_{\substack{T\in\tubes\\ T\subset T_\rho }}|T|\Big)\\
&\lesssim  C^p \rho^{\eps p}\delta^{\frac{(1-d)p}{p'}-\eps p}\sum_{T\in\tubes}|T|.
\end{split}
\end{equation}
Selecting $\rho$ sufficiently small compared to $\eps,$ $p$, and the implicit constants in \eqref{boundingLpNormPRd2} and \eqref{boundSecondTerm} (recall that these constants depend only on $n$ and $p$; crucially, they do not depend on $\rho$ or $C$), we obtain \eqref{LPBound}, which closes the induction and completes the proof. 
\end{proof}


\section{Proof of Conjecture \ref{mainConj}}\label{mainConjProofSec}
We are now ready to prove Conjecture \ref{mainConj}. First, we will reduce the situation where $L$ is analytic, and in particular supports a Frostman measure of the appropriate dimension. A reduction of this type first appeared in \cite{HKM}; we briefly summarize those arguments here. 

For convenience, we will suppose that $\beta>0$ (so in particular $d<n$). This is harmless, since if $\beta=0$ then we can replace $d$ by $d-1$ and $\beta$ by 1. Next, note that it suffices to establish a variant of \eqref{dimensionBound}, where the RHS has been replaced by $d+\beta-\eps$ for arbitrary $\eps>0$. 

Fix such an $\eps$.  Let $X = \bigcup_{\ell\in L}\ell$, and let $X_1 \supset X$ be a $G_\delta$ set with $\dim X_1 = \dim X$; the construction of such a set is explained, for example, in \cite{Fr}. Since $\ell\subset X_1$ for each $\ell\in L$, if we select $R$ sufficiently large then the set $L_1 = \{\ell \in L\colon |\ell \cap X_1 \cap B(0,R)| > 1\}$ satisfies $\dim L_1 \geq \dim L - \eps/2.$

Abusing notation, we will replace the set $X_1$ and the lines in $L_1$ and $L$ with their images under the re-scaling $x\mapsto x/R$. We still have $\dim L_1 \geq 2(d-1)+\beta - \eps/2$, but now 
\begin{equation}\label{defnL1}
L_1 = \{\ell \in L\colon |\ell \cap B(0,1)| > R^{-1}\}.
\end{equation}
The set $X_2 = X_1 \cap B(0,1)$ is $G_\delta$ and bounded, and $\ell\cap X_2$ contains an interval of length $\geq 1/R$ for each $\ell\in L_1.$

Next, we recall the following special case of Lemma 3.1 from \cite{HKM}. In the statement below, $A(n,1)$ is the set of lines in $\RR^n$ and if $Z\subset\ell$, then $\mathcal{H}^1_\infty(Z)=\inf \sum|J|$, where the infimum is taken over all coverings of $Z$ by countable collections of open intervals $\{J\}$. Note that the analyticity of a set $M\subset A(1,n)$ depends on the metric (or more generally, the topology), and the authors in \cite{HKM} discuss their choice of metric in Remark 2.2.  The metric on $\mathcal{L}_n\subset A(n,1)$ described in Section \ref{introSection} is compatible with that from \cite{HKM}.
\begin{lem}\label{HKMLem}
Let $Y\subset\RR^n$ be $G_\delta$ and bounded, and let $c>0$. Then the set of lines 
\[
\big\{\ell \in A(n, 1)\colon \mathcal{H}^1_\infty(\ell \cap Y)>c\big\}
\]
is analytic. 
\end{lem}

Let $L_2=\{\ell \in A(n,1) \colon \mathcal{H}^1_\infty(\ell \cap X_2) \geq 1/R\}$. Then $L_2$ is analytic, and $L_2\supset L_1$ (indeed, $\mathcal{H}^1_\infty(\ell \cap X_2) \geq 1/R$ for each $\ell\in L_1$, since $\ell \cap X_2$ contains an interval of length $\geq 1/R$), so $\dim L_2 \geq  2(d-1)+\beta-\eps/2$. At this point, we have replaced our original set $X=\bigcup_{\ell \in L}\ell$ by a new set $X_2$ whose Hausdorff dimension is no larger than that of $X$, and we have replaced the original set of lines $L$ by a set $L_2$ that is analytic (in the metric space $\mathcal{L}_n$), with $ \dim L_2\geq 2(d-1)+\beta-\eps/2$. In particular, $L_2$ supports a Frostman probability measure $\mathbb{P}$ that satisfies the following ball condition: There is a constant $C_0$ so that for all $r>0$,
\begin{equation}\label{ballConditionOnP}
\mathbb{P}(B) \leq C_0 r^{2(d-1)+\beta - \eps}\quad\textrm{for all balls}\ B\subset\mathcal{L}_n\ \textrm{of radius}\ r.
\end{equation}

To prove Conjecture \ref{mainConj}, it suffices to establish the estimate
\begin{equation}\label{dimBoundX2}
\dim X_2 \geq d+\beta - C_1\eps,
\end{equation}
for some constant $C_1=C_1(n)>0$; we will prove \eqref{dimBoundX2} with $C_1 = 7n$. 

Armed with Theorem \ref{maximalFnEstimateTubes}, the proof of \eqref{dimBoundX2} is a straightforward discretization argument. Let $k_0$ be a large integer that will be chosen below. Cover $X_2$ by a union $\bigcup_{k=k_0}^\infty \bigcup_{B\in\mathcal{B}_k}B$, where $\mathcal{B}_k$ is a collection of balls of radius $2^{-k}$, and $\sum_k 2^{-k(\dim X_2 +\eps)}\#\mathcal{B}_k<\infty$. 

For each line $\ell\in L_2$, there is an index $k=k(\ell)\geq k_0$ so that 
\[
\Big|\ell \cap \bigcup_{B\in\mathcal{B}_k}B\Big|\geq \frac{1}{100Rk^2},
\] 
where $R$ is the constant from \eqref{defnL1}. For each index $k$, define $L^{(k)} = \{\ell\in L_2\colon k(\ell)=k\}$. Then $\sum_{k\geq k_0}\mathbb{P}(L^{(k)})\geq 1,$ so there is an index $k_1\geq k_0$ with $\mathbb{P}(L^{(k_1)})\geq \frac{1}{100k_1^2}$. Define $L_3=L^{(k_1)}$, define $\delta = 2^{-k_1}$, and define $E_{\delta}=\bigcup_{B\in\mathcal{B}_{k_1}} B$.

In what follows, we will write $A\lessapprox B$ to mean $A\lesssim |\log\delta|^{C}B$, where $C$ is an absolute constant (in practice we can take $C\leq 10$). In particular, we have $\mathbb{P}(L_3)\gtrapprox 1$, and $|\ell \cap E_\delta|\gtrapprox 1$ for each $\ell\in L_3$.  

By dyadic pigeonholing, we can select a $2\delta$-separated set $L_3^\delta\subset L_3$ and a number $A\geq C_0^{-1}\delta^{\eps}$ so that $(A^{-1}\delta^{2(1-d)-\beta})(\#L_3^\delta)\gtrapprox 1$, and for each $\ell\in L_3^\delta$ we have
\begin{equation}\label{measureOfDeltaBall}
A^{-1}\delta^{2(d-1)+\beta}\leq \mathbb{P}(B(\ell, \delta))<2  A^{-1}\delta^{2(d-1)+\beta}.
\end{equation}
Comparing \eqref{ballConditionOnP} and \eqref{measureOfDeltaBall}, we see that if $B\subset\mathcal{L}_n$ is a ball of radius $r\geq\delta$, then
\begin{equation}
\#(L_3^\delta \cap B) \leq C_0 \delta^{-\eps} A(r/\delta)^{2(d-1)+\beta}.
\end{equation}
Let $L_4^\delta\subset L_3^\delta$ be obtained by randomly and independently selecting each $\ell\in L_3^\delta$ with probability $C_0^{-1}\delta^{2\eps}A^{-1}$. We claim that if $\delta>0$ is sufficiently small (depending on $\eps$), then there exists such a random selection $L_4^\delta$ with
\begin{equation}\label{sideOfL5}
\#L_4^\delta \geq \frac{1}{2}C_0^{-1}\delta^{2\eps}A^{-1}\#L_3^\delta \geq \delta^{2(1-d)-\beta+4\eps},
\end{equation}
which satisfies the ball condition 
\begin{equation}\label{ballsForL5}
\#(L_4^\delta \cap B) \leq (r/\delta)^{2(d-1)+\beta}\quad\textrm{for all balls}\ B\subset\mathcal{L}_n.
\end{equation}
This is a standard random sampling argument. The key insight is that it suffices to establish \eqref{ballsForL5} (with $(r/\delta)^{2(d-1)+\beta}$ strengthened to $\frac{1}{2}(r/\delta)^{2(d-1)+\beta}$) for a boundedly-overlapping family of balls of radius $r$, for each dyadic $\delta\leq r\leq 1$. There are only about $\delta^{-n}$ such balls, and the probability that \eqref{ballsForL5} fails for any particular ball is exponentially small in $\delta$. The existence of an appropriate random sample then follows from the union bound. See e.g. \cite[Section 2]{PYZ} for details. 

Finally, for each $\ell\in L_4^\delta$, let $\tilde\ell\subset\ell$ be a unit line segment in $B(0,1)$ with $|\tilde\ell \cap E_\delta|\gtrapprox 1$, and let $T(\ell)$ be the $\delta$-neighborhood of $\tilde\ell$. Let $\tubes=\{T(\ell)\colon \ell\in L_4^\delta\}$. Then $\tubes$ is a collection of $\delta$-tubes in $B(0,1)$ that satisfies \eqref{nonConcentration}, and $|T\cap E_\delta|\gtrapprox |T|$ for each $T\in\tubes$. By \eqref{sideOfL5} and H\"older's inequality, have that for $p\in (1,\infty)$, 
\[
\delta^{n-1}\delta^{2(1-d)-\beta+4\eps}\lesssim \sum_{T\in\tubes}|T|\lessapprox \sum_{T\in\tubes}|T\cap E_\delta|\leq |E_\delta|^{1/p'}\Big\Vert \sum_{T\in\tubes}\chi_T\Big\Vert_p\lesssim \delta^{(n-\dim X_2-\eps)/p'}\Big\Vert \sum_{T\in\tubes}\chi_T\Big\Vert_p,
\]
i.e.
\begin{equation}\label{lowerBdMaximal}
\Big\Vert \sum_{T\in\tubes}\chi_T\Big\Vert_p \gtrapprox \delta^{5\eps} \delta^{n-2d + 1 - \beta}\delta^{-(n-\dim X_2)/p'}.
\end{equation}
On the other hand, by Theorem \ref{maximalFnEstimateTubes}, if $p = \frac{d+\beta}{d+\beta-1}$ then 
\begin{equation}\label{upperBdMaximal}
\Big\Vert \sum_{T\in\tubes}\chi_T\Big\Vert_p \leq C\delta^{\frac{1-d}{p'}+\frac{n+1 -2d - \beta}{p}-\eps}.
\end{equation}
Comparing \eqref{lowerBdMaximal} and \eqref{upperBdMaximal} and simplifying, we conclude that $\delta^{\frac{\dim X_2 - d - \beta}{p'}}\lessapprox \delta^{-6\eps}$. Thus if $k_0$ is selected sufficiently large (which forces $\delta$ to be sufficiently small), we conclude that $\dim X_2 \geq d + \beta - 7(d+1)\eps$, which is \eqref{dimBoundX2}.
\bibliographystyle{plain}
\bibliography{lines_bib}

\begin{thebibliography}{10}

\bibitem{BCT}
Jonathan Bennett, Anthony Carbery, and Terence Tao.
\newblock On the multilinear restriction and {K}akeya conjectures.
\newblock {\em Acta Math.}, 196:261--302, 2006.

\bibitem{BG}
Jean Bourgain and Larry Guth.
\newblock Bounds on oscillatory integral operators based on multilinear
  estimates.
\newblock {\em Geom. Funct. Anal.}, 21:1239--1295, 2011.

\bibitem{CV}
Anthony Carbery and Stef\'an~Ingi Valdimarsson.
\newblock The endpoint multilinear {K}akeya theorem via the {B}orsuk-{U}lam
  theorem.
\newblock {\em J. Funct. Anal.}, 264:1643--1663, 2012.

\bibitem{Fr}
D.H. Fremlin.
\newblock {\em Measure Theory: Topological Measure Spaces (Vol. 4)}.
\newblock Torres Fremlin, 2003.

\bibitem{G}
Larry Guth.
\newblock The endpoint case of the {B}ennett-{C}arbery-{T}ao multilinear
  {K}akeya conjecture.
\newblock {\em Acta Math.}, 205:263--286, 2010.

\bibitem{He}
K.~H{\'e}ra.
\newblock Hausdorff dimension of {F}urstenberg-type sets associated to families
  of affine subspaces.
\newblock {\em Ann. Acad. Sci. Fenn.}, 44:903--923, 2019.

\bibitem{HKM}
K.~H{\'e}ra, T.~Keleti, and A.~M{\'a}th{\'e}.
\newblock Hausdorff dimension of unions of affine subspaces and of
  {F}urstenberg-type sets.
\newblock {\em J. Fractal Geom.}, 6:263--284, 2019.

\bibitem{DOb}
D.~Oberlin.
\newblock Exceptional sets of projections, unions of {$k$}-planes, and
  associated transforms.
\newblock {\em Israel J. Math.}, 202:331--342, 2014.

\bibitem{ROb}
R.~Oberlin.
\newblock Unions of lines in {$F^n$}.
\newblock {\em Mathematika}, 62:738--752, 2016.

\bibitem{PYZ}
M.~Pramanik, T.~Yang, and J.~Zahl.
\newblock A {F}urstenberg-type problem for circles, and a {K}aufman-type
  restricted projection theorem in {$\mathbb{R}^3$}.
\newblock arXiv:2207.02259, 2022.

\end{thebibliography}
\end{document}